\documentclass[12pt]{article}
\usepackage{mathtools} 
\usepackage{amsthm}
\usepackage{url}
\usepackage{amsmath} 
\usepackage{latexsym}  
\usepackage{amssymb}

\newcommand{\N}{\mathbb{N}}
\newcommand{\Z}{\mathbb{Z}}
\newcommand{\Cay}{\,\mathrm{Cay}}

\newtheorem{theorem}{\bf Theorem}
\newtheorem{proposition}{\bf Proposition}

\newtheorem{example}{\bf Example}

\newtheorem{lemma}{\bf Lemma}

\renewenvironment{proof}{{\it Proof.\/}}{\bf $\Box$\\}
\newenvironment{proofthm}{\vskip .1in{\it Proof\/}}{\bf $\Box$\\}

\begin{document}

\begin{center}
{\LARGE\bf Distance labellings of Cayley graphs of semigroups}
\end{center}

\begin{center}
{\sc Andrei Kelarev, Charl Ras, Sanming Zhou}
\end{center}

\begin{center}
School of Mathematics and Statistics\\
The University of Melbourne, Parkville, Victoria 3010, Australia\\
{\tt andreikelarev-universityofmelbourne@yahoo.com}\\
{\tt \{cjras,smzhou\}@ms.unimelb.edu.au}
\end{center}

\begin{abstract}
This paper establishes connections
between the structure of a semigroup and
the minimum spans of distance labellings
of its Cayley graphs.
We show that certain general
restrictions on the minimum spans are
equivalent to the semigroup being
combinatorial, and that other
restrictions are equivalent to the
semigroup being a right zero band.
We obtain a description
of the structure of all semigroups
$S$ and their subsets $C$ such that
$\Cay(S,C)$ is a disjoint union of
complete graphs, and show that
this description is also equivalent
to several restrictions on the
minimum span of $\Cay(S,C)$.
We then describe all graphs
with minimum spans satisfying the
same restrictions, and give
examples to show that a
fairly straightforward upper bound for
the minimum spans of the underlying
undirected graphs of Cayley
graphs turns out to be sharp even
for the class of combinatorial semigroups.
\end{abstract}

\noindent {\bf Keywords:}
distance labellings,
combinatorial semigroups,
bands,
Cayley graphs,
unions of complete graphs

\section{Introduction}\label{s:introduction}

This paper investigates distance labellings
of Cayley graphs of semigroups.
Throughout the paper the term \emph{graph}
means a finite directed graph without
multiple edges, but possibly with loops.
A graph $\Gamma = (V, E)$ is said to be
\textit{undirected} if and only if, for
every $(u, v) \in E$, the edge $(v, u)$
belongs to $E$. All our theorems will
involve undirected Cayley graphs only.
Let $\N$ be the set of all positive
integers, and let $k_1, k_2, \dots,
k_\ell \in \N$ for some $\ell\geq 1$.
A \textit{distance labelling}, or an
\emph{$L(k_1, k_2, \dots, k_\ell)$-labelling}
of an undirected graph $\Gamma = (V, E)$
is a mapping $f:V \rightarrow \N$ such
that $|f(u) - f(v)| \ge k_t$ for
$t= 1, 2, \dots, \ell$ and any
$u, v \in V$ with $d(u, v)=t$,
where $d(u, v)$ is the distance
between $u$ and $v$ in $\Gamma$.
The integer $f(v)$ is called the
\emph{label} assigned to $v$ under $f$,
and the difference between the largest and
the smallest labels is called the
\emph{span} of $f$. The minimum span
over the set of all $L(k_1, k_2, \dots,
k_\ell)$-labellings of $\Gamma$ is
denoted by $\lambda_{k_1, k_2, \dots,
k_\ell}(\Gamma)$.


Distance labellings and their
minimum spans have been investigated,
for example, in
\cite{Calamoneri2006,
ChangLuSanmingZhou2009,
KingRasSanmingZhou2010,
KingLiSanmingZhou2014,
SanmingZhou2006}.
The problem of finding
$\lambda_{k_1, k_2, \dots,
k_\ell}(\Gamma)$ is motivated
by the radio channel assignment problem and by
the study of the scalability of optical
networks (cf.
\cite{ChangLuSanmingZhou2007,
SanmingZhou2008}).
The problem is also important for the
theory of classical graph colourings:
for any graph $\Gamma$,
it is clear that the minimum span
$\lambda_{1, \ldots, 1}(\Gamma)$
is one less than the chromatic number
of the $\ell$-th power of $\Gamma$.
This relation is valuable, since the problem
of studying the chromatic numbers of powers
of graphs is a significant problem in
graph theory (see, for example,
\cite{{AlonMohar2002}}). For illustrating
examples of distance labellings and their
minimum spans the reader can turn to
Examples~4 and~5 in Section~2 of our
paper.


The notion of a Cayley graph was
introduced by Arthur Cayley to explain
the structure of groups defined by a set
of generators and relations. The
literature on Cayley graphs
of groups is vast and includes many
deep results related to structural and
applied graph theory; see for instance
\cite{BrankovicMillerPlesnikRyanSiran1998,
ChangLuSanmingZhou2007,
LakshmivarahanJwoDhall1993,
LiPraeger1999,
Praeger1999,
SanmingZhou2006}.

Cayley graphs of semigroups, which we
define next, are
also very well known and have been
studied since at least the 1960's
\cite{Denes1967,Zelinka1981}.
Let $S$ be a semigroup, and let $C$
be a nonempty subset of $S$. The
\textit{Cayley graph} $\Cay(S, C)$ of
$S$ with \emph{connection set} $C$
is defined as the graph with vertex
set $S$ and edge set $E = E(S, C)$
consisting of those ordered pairs $(u,
v)$ such that $c u = v$ for some $c \in
C$. If the adjacency relation of
$\Cay(S, C)$ is symmetric, then it
becomes an undirected graph.
Cayley graphs are the most important
class of graphs associated with
semigroups, as they have valuable
applications
(cf.~\cite{ChangLuSanmingZhou2007,
KelarevRyanYearwood2009cayley,
Knauer2011})
and are related to automata theory
(cf.~\cite{Kelarev2004cayley,
Kelarev:gaa}).
Many interesting results on
Cayley graphs of various classes
of semigroups have been obtained
recently, for example, in
\cite{KhosraviKhosraviKhosravi2014,
KhosraviMahmoudi2010,
Knauer2011,
LuoHaoClarke2011,
PanmaChiangmaiKnauerArworn2006,
WangLi2013}.

A semigroup is said to be
\textit{combinatorial} if all of
its subgroups are singletons (see
\cite{AbawajyKelarev2013combinatorial,
KelarevYearwoodWatters2011combinatorial}
for recent results on this
class of semigroups).
Cayley graphs of combinatorial
semigroups have not been considered
in full generality; only small
subclasses of this class have been
investigated. A \textit{band} is
a semigroup entirely consisting of
idempotents, i.e., elements $x$
satisfying $x = x^2$. The class of
combinatorial semigroups contains all
bands and all Brandt semigroups over
groups of order one.  Bands play
crucial roles in the structure
theory of semigroups (see
\cite{Howie:fst} and
\cite{Kelarev:gaa}). Cayley
graphs of bands were explored
in \cite{FanZeng2007,
GaoLiuLuo2011,
Jiang2004,
KhosraviKhosravi2013,
KhosraviKhosravi2014,
KhosraviKhosraviKhosravi2015},
and Cayley graphs of Brandt semigroups
were studied in \cite{HaoGaoLuo2011,
KhosraviKhosravi2012}.

Our first main theorem shows that
certain general restrictions on a
certain formula expressing the span
of a Cayley graph of a semigroup can
force the semigroup to be combinatorial
(Theorem~\ref{t:span_combinatorial})
or to be a right zero band
(Theorem~\ref{t:span_right_zero_band}),
which are both well-known semigroup classes.

In addition to the results above, in
Theorem~\ref{t:span_cayley_union_of_complete}
we describe all semigroups
$S$ and their subsets $C$ such that
the Cayley graph $\Cay(S, C)$
satisfies the restrictions that occur
in our first two theorems. This yields the
first description of the structure
of all semigroups that possess Cayley
graphs that are disjoint unions of
complete graphs. An abstract
characterisation of all Cayley graphs of
semigroups which are disjoint unions of
complete graphs was obtained in
\cite{Kelarev2004inverse}
in the language of equalities that hold
for all elements of the semigroup.
For several special classes of
semigroups, a number of improvements
and alternative element-wise
characterisations have been obtained
in \cite{LuoHaoClarke2011}
and~\cite{WangLi2013}.
However, the structure of such
semigroups has not been characterised
before, except in the special case of semigroups
with bipartite Cayley graphs \cite{Kelarev2004inverse}.

In Theorem~\ref{t:span=F(k1)},
we give a characterisation of
all graphs whose minimum spans can be determined
by a formula with some of the
restrictions that occur in the first
three theorems.

Finally, we give examples to
show that a straightforward upper bound
for the minimum spans of the underlying
undirected graphs of Cayley
graphs turns out to be sharp even for
the class of combinatorial semigroups.

\section{Main results}
\label{s:main_results}

Our first main result shows that
several restrictions on the minimum
span of a Cayley graph are equivalent
to the semigroup being combinatorial. As
mentioned before, all our theorems will
involve undirected Cayley graphs only.
For any set or semigroup $Q$, we use the
notation $|Q|$ for the cardinality of $Q$.

\begin{theorem}\label{t:span_combinatorial}
For any finite semigroup $S$,
the following conditions are
equivalent:

\begin{itemize}
\item[\rm(i)]
$S$ is combinatorial;

\item[\rm(ii)]
there exist a function
$F:\N \rightarrow \N$ and an
integer $\ell > 1$ such that,
for each nonempty subsemigroup $T$
of $S$ and every nonempty subset
$C$ of $T$, if $\Cay(T, C)$
is undirected then
\begin{equation}\label{eq:comb:lambda_le_F(k1)}
\lambda_{k_1, \ldots, k_\ell}
(\Cay(T, C)) \le F(k_1);
\end{equation}

\item[\rm(iii)]
there exist a function
$F:\N \rightarrow \N$ and
an integer $\ell > 1$ such that,
for each nonempty subsemigroup $T$
of $S$ and every nonempty subset
$C$ of $T$, if $\Cay(T, C)$
is undirected, then
\begin{equation}
\lambda_{k_1, \ldots, k_\ell}
(\Cay(T, C)) = F(k_1);
\end{equation}

\item[\rm(iv)]
there exists $\ell > 1$ such that,
for each nonempty subsemigroup $T$
of $S$ and every nonempty subset
$C$ of $T$, if $\Cay(T, C)$
is undirected, then
\begin{equation}
\lambda_{k_1, \ldots, k_\ell}(\Cay(T, C)) =
(|Cc| - 1) k_1
\end{equation}
for a fixed element $c$ in $C$;

\item[\rm(v)]
for each nonempty subsemigroup $T$
of $S$ and every nonempty subset $C$
of $T$, if $\Cay(T, C)$ is
undirected, then the equality
\begin{equation}\label{eq:comb:lambda=(|Cc|-1)k1}
\lambda_{k_1, \ldots, k_\ell}(\Cay(T, C)) =
(|Cc| - 1) k_1
\end{equation}
holds for every integer $\ell > 1$,
all $k_1, \ldots, k_\ell \in \N$,
and every element $c$ is $C$.
\end{itemize}
\end{theorem}

A band $B$
is called a \textit{right zero band},
if it satisfies the identity
$x y = y$, for all $x, y \in B$.
Our second result establishes
several conditions on the minimum
span that are equivalent to the semigroup
being a right zero band.

\begin{theorem}\label{t:span_right_zero_band}
For any finite semigroup $S$,
the following conditions are
equivalent:

\begin{itemize}
\item[\rm(i)]
$S$ is a right zero band;

\item[\rm(ii)]
there exist a function
$F:\N \rightarrow \N$ and
an integer $\ell > 1$
such that, for each nonempty
subset $C$ of $S$, $\Cay(S, C)$
is undirected and
\begin{equation}
\lambda_{k_1, \ldots, k_\ell}
(\Cay(S, C)) \le F(k_1);
\end{equation}

\item[\rm(iii)]
there exist a function
$F:\N \rightarrow \N$ and
an integer $\ell > 1$ such
that, for each nonempty subset
$C$ of $S$, $\Cay(S, C)$
is undirected and
\begin{equation}
\lambda_{k_1, \ldots, k_\ell}
(\Cay(S, C)) = F(k_1);
\end{equation}

\item[\rm(iv)]
there exists $\ell > 1$ such
that, for each nonempty subset
$C$ of $S$, $\Cay(S, C)$ is
undirected and
\begin{equation}
\lambda_{k_1, \ldots, k_\ell}
(\Cay(S, C)) = (|Cc| - 1) k_1,
\end{equation}
where $c$ is a fixed element in $C$;

\item[\rm(v)]
for each nonempty subset $C$ of $S$,
$\Cay(S, C)$ is
undirected and the equality
\begin{equation}\label{eq:lambda=(|Cc|-1)k_1}
\lambda_{k_1, \ldots, k_\ell}(\Cay(S, C)) =
(|Cc| - 1) k_1
\end{equation}
holds for every integer $\ell > 1$,
all $k_1, \ldots, k_\ell \in \N$,
and all elements $c$ of $C$.
\end{itemize}
\end{theorem}

The next theorem gives a
description of the structure
of all semigroups whose Cayley
graphs are disjoint unions of
complete graphs, where it is assumed
that the complete graphs contain
all loops. At the same
time it shows that this condition
is equivalent to certain
restrictions on the minimum span.
If $S$ is a semigroup with a
subset $C$, then the subsemigroup
generated by $C$ in $S$ is
denoted by~$\langle C \rangle$.
Several conditions equivalent to
the definition of a left group are
given in Section~\ref{s:preliminaries}
(see Lemma~\ref{l:right-groups}).

\begin{theorem}\label{t:span_cayley_union_of_complete}
For any semigroup $S$ and
any nonempty subset $C$ of $S$,
the following conditions are
equivalent:

\begin{itemize}
\item[\rm(i)]
$CS = S$, $\langle C \rangle$
is a completely simple semigroup,
and for each $c \in C$ the
set $C c$ is a left group
and a left ideal of
$\langle C \rangle$;

\item[\rm(ii)]
there exist a function $F:\N
\rightarrow \N$ and an integer
$\ell > 1$ such that, for all
$k_1, \dots, k_\ell \in \N$,
$\Cay(S, C)$ is undirected and
\begin{equation}
\lambda_{k_1, \dots, k_\ell}
(\Cay(S, C)) \le F(k_1);
\end{equation}

\item[\rm(iii)]
there exist a function $F:\N
\rightarrow \N$ and an integer
$\ell > 1$ such that, for all
$k_1, \dots, k_\ell \in \N$,
$\Cay(S, C)$ is undirected and
\begin{equation}
\lambda_{k_1, \dots, k_\ell}
(\Cay(S, C)) = F(k_1);
\end{equation}

\item[\rm(iv)]
there exists $\ell > 1$
such that, for all
$k_1, \dots, k_\ell \in \N$,
$\Cay(S, C)$ is undirected and
\begin{equation}\label{eq:lambda=(|Cc|_1)k_1}
\lambda_{k_1, \ldots, k_\ell}(\Cay(S, C)) =
(|Cc| - 1) k_1,
\end{equation}
where $c$ is a fixed
element in $C$;

\item[\rm(v)]
$\Cay(S, C)$
is undirected, and for every
integer $\ell > 1$ and all
$k_1, \dots, k_\ell \in \N$,
the equality
\begin{equation}\label{eq:complete:lambda=(|Cc|_1)k1}
\lambda_{k_1, \ldots, k_\ell}(\Cay(S, C)) =
(|Cc| - 1) k_1,
\end{equation}
holds for all elements
$c$ of $C$;

\item[\rm(vi)]
$\Cay(S, C)$ is a disjoint
union of complete graphs.
\end{itemize}
\end{theorem}

\begin{example}\label{ex:left_groups_are_essential}
{\rm
The requirement that $C c$ be a left group
in condition~(i) of
Theorem~\ref{t:span_cayley_union_of_complete}
cannot be omitted.
For example, let $\Z_{3} = \{e, g, g^{2}\}$
be the cyclic group of order $3$,
and let $C = \{e, g\}$. Then
$\langle C \rangle = \Z_3$ is
a completely simple semigroup and
$C\Z_3 = \Z_3$, but $\Cay(\Z_3, C)$ is
not a disjoint union of complete graphs.
}\end{example}

Suppose that $G$ is a group, $I$ and
$\Lambda$ are nonempty sets, and
$P=[p_{\lambda i}]$ is a
\mbox{$(\Lambda\times I)$-matrix} with
entries $p_{\lambda i}\in G$ for all
$\lambda\in\Lambda$, $i\in I$. The
\textit{Rees matrix semigroup}
$M(G; I, \Lambda; P)$ over $G$ with
\textit{sandwich-matrix} $P$ consists of
all triples $(h; i, \lambda)$, where $i \in
I$, $\lambda\in\Lambda$, and $h\in G$,
with multiplication defined by the rule
\begin{equation}\label{eq:Rees_matrix_product}
(h_1; i_1, \lambda_1)(h_2; i_2, \lambda_2)
=(h_1 p_{\lambda_1 i_2} h_2; i_1, \lambda_2).
\end{equation}

A semigroup is said to be
\textit{completely simple} if it has no
proper ideals and has an idempotent
minimal with respect to the natural
partial order defined on the set of all
idempotents by $e \le f \Leftrightarrow
e f = f e = e$.
It is well known that every completely
simple semigroup is isomorphic to a Rees
matrix semigroup $M(G; I, \Lambda; P)$ over
a group $G$ (see \cite{Howie:fst},
Theorem~3.3.1). Conversely, every
semigroup $M(G; I, \Lambda; P)$ is
completely simple.

A band $B$ is called a
\textit{left zero band}, if
it satisfies the identity $x y = x$,
for all $x, y \in B$. Recall also that band $B$
is called a \textit{right zero band},
if it satisfies the identity
$x y = y$, for all $x, y \in B$.

\begin{example}\label{ex:coincidence_is_essential}
{\rm
The requirement that $C c$ be a left ideal of
$\langle C \rangle$ in condition~(i) of
Theorem~\ref{t:span_cayley_union_of_complete}
cannot be omitted,
even if it is still required to be a left
group. For example, let $I = \{i_1, i_2\}$ be
a left zero band, and let $G = \{e, g\}$
be the cyclic group of order $2$.
Then the direct
product $G \times I$ is isomorphic
to the completely simple semigroup
$M(G; I, \Lambda; P)$, where
$\Lambda = \{\lambda\}$, and
$P = [e, e]$.
Clearly, $G \times I$ is generated by
the set $C = \{(g, i_1), (g, i_2)\}$.
Take $c = (g, i_1)$. Then
$C c = \{ (e, i_1), (e, i_2)\}$ is a
left group, but $\Cay(G \times I, C)$ is
not a disjoint union of complete graphs.
}\end{example}

The next example shows that the
complete graphs that occur in the
disjoint union of condition~(vi) of
Theorem~\ref{t:span_cayley_union_of_complete}
may have different cardinalities.
For $n \in \N$, the complete graph
with $n$ vertices is denoted by~$K_n$.

\begin{example}\label{ex:left_zero_band_with_zero_adjoined}
{\rm
Let $B$ be a finite left zero band,
and let $B^0 = B \cup \{\theta\}$ be the
semigroup obtained by adjoining zero
$\theta$ to $B$. Then it is easily seen
that
\begin{equation}
\Cay(B^0, B) = K_{|B|} \cup K_1
\end{equation}
is a disjoint union of two complete
graphs, where $K_{|B|}$ and $K_1$ are
the subgraphs induced in $\Cay(B^0, B)$
by the sets $B$ and $\{\theta\}$,
respectively.
}\end{example}

The last main result of this paper
describes all graphs whose minimum
spans satisfy some of the restrictions
that occur in the previous theorems.

\begin{theorem}             \label{t:span=F(k1)}
For any finite undirected graph
$\Gamma=(V, E)$, the following
conditions are equivalent:

\begin{itemize}
\item[\rm(i)]
there exist a function $F: \N
\rightarrow \N$ and an
integer $\ell > 1$
such that, for all
$k_1, \dots, k_\ell \in \N$,
the minimum span of $\Gamma$
satisfies the inequality
\begin{equation}\label{eq:lambda<=F(k1)}
\lambda_{k_1, \dots, k_\ell}(\Gamma)
\le F(k_1);
\end{equation}

\item[\rm(ii)]
there exist a function
$F: \N \rightarrow \N$
and an integer $\ell > 1$
such that, for all
$k_1, \dots, k_\ell \in \N$,
the minimum span of $\Gamma$ is
determined by the formula
\begin{equation}\label{eq:lambda=F(k1)}
\lambda_{k_1, \ldots, k_\ell}(\Gamma) =
F(k_1);
\end{equation}

\item[\rm(iii)]
$\Gamma$ is a disjoint union of complete
graphs.
\end{itemize}
\end{theorem}

We conclude this section with examples
concerning the following straightforward
upper bound on the minimum span of any
undirected graph $\Gamma = (V, E)$:
\begin{equation}           \label{eq:upper_bound_underlying}
\lambda_{k_1, \ldots, k_\ell}(\Gamma)
\le
(|V| - 1) \max\{k_1, \ldots, k_\ell\}.
\end{equation}
This inequality holds because
the labelling that assigns
$(i - 1) \max\{k_1, \ldots, k_\ell\} + 1$
to the $i$-th vertex (for any order
of vertices) of $\Gamma$ is an
$L(k_1, \ldots, k_\ell)$-labelling
with span
$(|V| - 1) \max\{k_1, \ldots, k_\ell\}.$

The \textit{underlying undirected graph}
$\Gamma^*$ of a directed graph
$\Gamma = (V, E)$ is the graph with
the same set $V$ of vertices and
with all undirected edges $\{u, v\}$
such that $(u, v)$ or $(v, u)$ is
a directed edge
of $\Gamma$. The following examples
show that there exist combinatorial
semigroups such that, for the
underlying undirected graphs of
their Cayley graphs, the upper bound
(\ref{eq:upper_bound_underlying})
is sharp.

\begin{example}\label{ex:left_zero_band}
{\rm
Let $B$ be a left zero band, and let
$C$ be a nonempty subset of $B$. Then
\begin{equation}\label{eq:left_zero_band_1}
\lambda_{k_1, \ldots, k_\ell}(\Cay(B, C)^*) =
k_1(|C| - 1) + \max\{k_1, k_2\}
+ k_2(|B| - |C| - 1).
\end{equation}
Indeed, it follows from the definition
of a left zero band that the set of
edges of $\Cay(B, C)$ is the set of
all edges $(b, c)$, for all $b \in B$,
$c \in C.$ Therefore, for $x, y \in B$,
$x \ne y$, we get
$$d(x, y) = \left\{
\begin{array}{ll}
1 & \mbox{ if } x \in C \mbox{ or } y \in C \\
2 & \mbox{ if } x, y \in B \setminus C.
\end{array}
\right. $$
We can assign the labels $1$, $1 + k_1$,
$\dots$, $1 + k_1(|C| - 1)$ to the
elements of $C$, and the labels
$1 + k_1(|C| - 1) + \max\{k_1, k_2\}$,
$1 + k_1(|C| - 1) + \max\{k_1, k_2\} + k_2$,
$\dots$,
$1 + k_1(|C| - 1) + \max\{k_1, k_2\}
+ k_2(|B| - |C| - 1)$
to the elements of $B \setminus C$.
This defines an
$L(k_1, \dots, k_\ell)$-labelling
of $\Cay(B, C)$. The span of this
labelling is equal to the right-hand
side of (\ref{eq:left_zero_band_1}),
and it is clear that this is
the minimum span. In the case
where $k_1 = k_2$, equality
(\ref{eq:left_zero_band_1})
shows that the upper bound
(\ref{eq:upper_bound_underlying})
is sharp.
}\end{example}

The following example covers all
combinatorial semigroups with zero.
Notice that, for any combinatorial
semigroup $S$, the semigroup $S^0 = S \cup
\{\theta\}$ with zero $\theta$ adjoined
in a standard fashion is combinatorial
as well.

\begin{example}\label{ex:semigroup_with_zero}
{\rm
Let $S$ be a finite semigroup with
zero $\theta$, and let $C = \{\theta\}$.
Then  the set of edges of
$\Cay(S, C)$ coincides with the set
$\{(x,\theta) \mid x \in S\}$. Therefore
\begin{equation}\label{eq:Cayley_semigroup_with_zero_1}
\lambda_{k_1, \ldots, k_\ell}(\Cay(S, C)^*) =
k_1 + (|S| - 2) k_2.
\end{equation}
Indeed, for $x, y \in B$, $x \ne y$,
we have
$$d(x, y) = \left\{
\begin{array}{ll}
1 & \mbox{ if } x = \theta \mbox{ or } y = \theta \\
2 & \mbox{ if } x, y \in S \setminus \{\theta\}.
\end{array}
\right. $$
Therefore we can assign the label $1$ to
$\theta$ and the labels
$1 + k_1$, $1 + k_1 + k_2$, $1 + k_1 + 2
k_2$, $\dots$, $1 + k_1 + (|S| - 2) k_2$
to the nonzero elements of $S$.
This defines an
$L(k_1, \dots, k_\ell)$-labelling
of $\Cay(S, C)$ with the minimum
possible span. This shows that
$\lambda_{k_1, \dots, k_\ell}(\Cay(S, C))$
is equal to the right-hand side of
(\ref{eq:Cayley_semigroup_with_zero_1}).
In the case where $k_1 = k_2$, equality
(\ref{eq:Cayley_semigroup_with_zero_1})
shows that the bound
(\ref{eq:upper_bound_underlying})
is sharp.
}\end{example}

\section{Preliminaries}
\label{s:preliminaries}

We use standard terminology on graphs,
Cayley graphs, groups and semigroups,
and refer to
\cite{Howie:fst,Kelarev:gaa,Knauer2011}
for more detailed explanations. The
following notation and definitions are
required for the proofs. A
semigroup $G$ is \textit{periodic} if,
for each $g \in G$, there exist positive
integers $m, n$ such that $g^m = g^{m +
n}$. Obviously, every finite semigroup
is periodic. A semigroup is said to be
\textit{right} (\textit{left})
\textit{simple} if it has no proper
\textit{right} (\textit{left})
\textit{ideals}. A semigroup is
\textit{left} (\textit{right})
\textit{cancellative} if $x y = x z$
(respectively, $y x = z x$) implies $y =
z$, for all $x, y, z\in S$. A semigroup
is called a \textit{right}
(\textit{left}) \textit{group} if it is
right (left) simple and left (right)
cancellative.

\begin{lemma}                     \label{l:right-groups}
{\rm(\cite[Lemma~3.1]{KelarevPraeger2003})}
For any periodic semigroup $S$,
the following are equivalent:
\begin{itemize}
\item[\rm(i)]
$S$ is right (left) simple;
\item[\rm(ii)]
$S$ is a right (left) group;
\item[\rm(iii)]
$S$ is isomorphic to the direct product of
a right (left) zero band and a group;
\item[\rm(iv)]
$S$ is a union of several of its left (right)
ideals and each of these ideals is a group.
\end{itemize}
\end{lemma}

A band $B$ is called a
\textit{rectangular band}
if it satisfies the identity
$x y x = x$, for all $x, y \in B$.

\begin{lemma}            \label{l:rectangular=left_x_right}
{\rm(\cite[Theorem~1.1.3]{Howie:fst})}
Let $B$ be a rectangular band.
Then there exists a left zero band
$L$ and a right zero band $R$ such
that $B$ is isomorphic to the
direct product $L \times R$.
\end{lemma}

It is very well known, and also
follows immediately from
Lemma~\ref{l:rectangular=left_x_right},
that every rectangular band $B$
satisfies the identity $x y z = x z$,
for all $x, y, z \in B$.

Let $G$ be a group, $T=M(G; I, \Lambda; P)$,
and let $i \in I$, $\lambda \in \Lambda$.
Then we put
$$T_{* \lambda} = \{(h; j, \lambda) \mid h \in G, j \in I\},$$
$$T_{i *}=\{(h; i, \mu) \mid h \in G, \mu \in \Lambda\},$$
$$T_{i \lambda}=\{(h; i, \lambda) \mid h \in G\}.$$

\begin{lemma}                  \label{l:Rees_matrix}
{\rm(\cite[Lemma~3.2]{KelarevPraeger2003})}
Let $G$ be a group, and let
$T = M(G;I,\Lambda; P)$ be a
completely simple semigroup.
Then, for all $i, j \in I$,
$\lambda, \mu \in \Lambda$,
and $t = (h; i, \lambda) \in T$,
\begin{itemize}
\item[\rm(i)]
the set $T_{* \lambda}$ is a minimal
nonzero left ideal of $T$;
\item[\rm(ii)]
the set $T_{i *}$ is a minimal
nonzero right ideal of $T$;
\item[\rm(iii)]
$T t = T_{* \mu} t = T_{* \lambda}$;
\item[\rm(iv)]
$t T = t T_{j *} = T_{i *}$;
\item[\rm(v)]
$t \in T t \cap t T = T_{i\lambda}$;
\item[\rm(vi)]
the set $T_{i\lambda}$ is a
left ideal of $T_{i *}$ and a
right ideal of $T_{* \lambda}$;
\item[\rm(vii)]
the set $T_{i\lambda}$ is a maximal
subgroup of $T$ isomorphic to $G$;
\item[\rm(viii)]
each maximal subgroup of $T$
coincides with $T_{j \mu}$,
for some $j \in I$,
$\mu \in \Lambda$;
\item[\rm(ix)]
$M(G; I, \Lambda; P)$ is
a right (left) group if
and only if $|I| = 1$
(respectively, $|\Lambda| = 1$);
\item[\rm(x)]
if $T = M(G; I, \Lambda; P)$,
then each $T_{* \lambda}$ is
a left group, and each
$T_{i *}$ is a right group.
\end{itemize}
\end{lemma}

It is well known that the Cayley graph
$\Cay(G, C)$ of a group $G$ is symmetric
or undirected if and only if $C = C^{-1}$,
that is $c \in C$ implies $c^{-1} \in C$.
Undirected Cayley graphs of
semigroups have been characterised in
\cite{Kelarev2002undirected}.

\begin{proposition}         \label{p:Cayley_undirected}
{\rm(\cite[Lemma~4]{Kelarev2002undirected})}
Let $S$ be a semigroup, and let $C$
be a subset of $S$, which generates
a periodic subsemigroup
$\langle C \rangle$. Then the
following conditions are equivalent:
\begin{itemize}
\item[\rm(i)]
$\Cay(S, C)$ is undirected;
\item[\rm(ii)]
$C S = S$, the semigroup
$\langle C \rangle = M(H; I, \Lambda; P)$
is completely simple and, for each
$(g; i, \lambda) \in C$
and every $j \in I$, there exists
$\mu \in \Lambda$ such that
$(p_{\lambda j}^{-1} g^{-1}
p_{\mu i}^{-1}; j, \mu) \in C$.
\end{itemize}
\end{proposition}

A graph $\Gamma$ is said to be
\textit{connected} if its underlying
undirected graph is connected. If,
for each pair of vertices $x, y$ of
$\Gamma$, there exists a directed path
from $x$ to $y$, then $\Gamma$ is said
to be \textit{strongly connected}.
A (strong) connected component of
$\Gamma$ is a maximal strongly connected
subgraph of $\Gamma$. Obviously,
$\Gamma$ is a vertex-disjoint union of
its connected components.

\begin{lemma}\label{l:connected_component}
{\rm(\cite[Lemma~5.1]{KelarevPraeger2003})}
Let $S$ be a semigroup, $C$ a subset
of $S$, and $x \in S$. Let $C_x$ be
the set of all vertices $y$ of
$\Cay(S, C)$ such that
there exists a directed path from $x$
to $y$. Then $C_x$ is equal to the
right coset $\langle C \rangle x$.
\end{lemma}

\section{Proofs of the main results}
\label{s:proofs}

\begin{proofthm} {\it of
Theorem~\ref{t:span=F(k1)}.}
The implication (ii)$\Rightarrow$(i)
is obvious. Let us prove implications
(i)$\Rightarrow$(iii) and
(iii)$\Rightarrow$(ii).

(i)$\Rightarrow$(iii):
Suppose to the contrary that condition
(i) holds, but (iii) is not satisfied.
This means that there exist a function
$F: \N \rightarrow \N$ and an integer
$\ell > 1$ such that
(\ref{eq:lambda<=F(k1)}) holds for all
$k_1, \dots, k_\ell \in \N$, but the
graph $\Gamma$ is not a disjoint union
of complete graphs. Clearly, any graph
is a disjoint union of complete graphs
if and only if $d(u, v) \in \{0, 1,
\infty\}$ for every pair of
vertices $u, v$ of the graph. It
easily follows that $\Gamma$ contains
a pair of vertices $u, v$ at distance
$d(u, v) = 2$ since the diameter of
every connected component is not $1$.
Fix any value $k_1 \in \N$ and take
$k_2 = F(k_1) + 1$. For any
$L(k_1, \dots, k_\ell)$-labelling
$f$ of $\Gamma$ with span
$\lambda_{k_1, \dots, k_\ell}$,
we get $|f(u) - f(v)|
\ge k_2 > F(k_1)$. The maximality of
$\lambda_{k_1, \dots, k_\ell}$ implies
that $|f(u) - f(v)| \le \lambda_{k_1,
\dots, k_\ell}(\Gamma) \le F(k_1)$. This
contradiction shows that our assumption
was wrong and (iii) must have been
satisfied.

(iii)$\Rightarrow$(ii):
Suppose that $\Gamma$ is a disjoint
union of complete graphs. Choose a
connected component of $\Gamma$ with
the largest number of vertices, and
denote the vertices of this component
by $v_1,\dots, v_n$, where $n \in \N$.
Then the mapping $f$ defined by
$f(v_i) = k_1 (i - 1) + 1$, for
$i = 1, \dots, n$, is an $L(k_1, \dots,
k_\ell)$-labelling of the connected
component. Its span is equal to
$|f(v_n)-f(v_1)|=|(k_1(n-1)+1)-1|=k_1 (n - 1)$.
Similar labellings can be defined for all other
connected components of $\Gamma$,
since each of them has at most $n$
vertices. This defines an
$L(k_1, \dots, k_\ell)$-labelling of
$\Gamma$ with span $k_1 (n - 1)$.
Taking the function $F(x) = x (n - 1)$,
we see that equality
(\ref{eq:lambda=F(k1)})
is always valid for $F$.
This means that (ii) holds,
which completes the proof.
\end{proofthm}

\begin{proofthm} {\it of
Theorem~\ref{t:span_cayley_union_of_complete}.}
Implications
(v)$\Rightarrow$(iv)$\Rightarrow$(iii)$\Rightarrow$(ii)
are obvious. Implication (ii)$\Rightarrow$(vi)
follows from Theorem~\ref{t:span=F(k1)}.
Let us prove implications
(i)$\Rightarrow$(vi),
(vi)$\Rightarrow$(i) and
(vi)$\Rightarrow$(v).

(i)$\Rightarrow$(vi):
Suppose that condition (i) holds.
Then there exist a group $G$, sets
$I$, $\Lambda$ and $\Lambda \times I$
sandwich matrix $P$ over $G$ such
that the subsemigroup $T = \langle C
\rangle$ is isomorphic to the
completely simple semigroup
$M(G; I, \Lambda; P)$.

Take any element $s$ in $S$ and
consider any pair of vertices
$t_1 s, t_2 s$, where $t_1, t_2 \in T$.
Given that $S = C S$, we can find
$c \in C$ and $s_1 \in S$ such that
$s = c s_1$. Since $c \in T$, there
exist $g_c \in G$, $i_c \in I$ and
$\lambda_c \in \Lambda$ such that
$c = (g_c; i_c, \lambda_c)$.
Likewise, $t_j = (g_j; i_j, \lambda_j)$
for $j = 1, 2$, $g_j \in G$,
$i_j \in I$ and
$\lambda_j \in \Lambda$.
Since $S = C S$, we can also find
$c_1 \in C$ and $s_2 \in S$ such that
$t_1 c = c_1 s_2$. Since the set $C c_1$
is a left ideal of $T$, it follows that
$C t_1 c = C c_1 s_2$ is also a left
ideal of $T$.
Lemma~\ref{l:Rees_matrix}(ii) implies
that $C_{i_2*} t_1 c = C t_1 c
\cap T_{i_2*}$. Hence $C_{i_2*} t_1 c$
is a left ideal of $T_{i_2*}$, because
$C t_1 c$ is a left
ideal of $T$. In view of conditions
(i) and (ii) of Lemma~\ref{l:Rees_matrix},
$C_{i_2*} t_1 c$ is contained in
$T_{i_2 \lambda_c}$. Hence
$C_{i_2*} t_1 c$ is a left ideal of
$T_{i_2 \lambda_c}$. However,
$T_{i_2 \lambda_c}$ is a subgroup
of $T$ by
Lemma~\ref{l:Rees_matrix}(vii),
and groups do not have proper
left ideals. Therefore
$C_{i_2*} t_1 c = T_{i_2 \lambda_c}$,
and so $t_2 c = c_2 t_1 c$ for some $c_2
\in C_{i_2 *}$. Thus, we get $t_2 s =
t_2 c s_1 = c_2 t_1 c s_1 = c_2 t_1 s$.
This means that $(t_1 s, t_2 s)$ is
an edge of $\Cay(S, C)$.
Since $t_1 s$ and $t_2 s$ were chosen
arbitrarily, it follows that
the set $T s$ induces a complete
subgraph of $\Cay(S, C)$.

Clearly, $S$ is a union of the
sets $T s$, for $s \in S$. Suppose that
two sets $T s_1$ and $T s_2$ are not
disjoint for some $s_1, s_2 \in S$.
Then there exists $x \in T s_1 \cap T
s_2$. Choose any $x_1 \in T s_1$ and
$x_2 \in T s_2$. Since $T s_1$ induces a
complete subgraph in $\Cay(S, C)$, we
see that $(x, x_1)$ is an edge of
$\Cay(S, C)$. Similarly, $(x, x_2)$ is
an edge of $\Cay(S, C)$. Hence
$x_1, x_2 \in C x \subseteq T x$.
However, $T x$ also induces a complete
subgraph of $\Cay(S, C)$. It follows
that $(x_1, x_2)$ is an edge of
$\Cay(S, C)$.
Lemma~\ref{l:connected_component}
implies that $T x = T s_1 = T s_2$.

Since the semigroup $S$ is assumed
to be finite, it follows that for
some $n \in \N$ and
$s_1, \dots, s_n \in S$ the
graph $\Cay(S, C)$ is a disjoint union
of the complete graphs induced by the
sets $T s_1, \dots, T s_n$. Thus,
condition (vi) is satisfied.

(vi)$\Rightarrow$(i):
Suppose that condition (vi) holds,
 that is, $\Cay(S, C)$ is a disjoint
union of complete graphs. Then
$\Cay(S, C)$ is undirected, and so
Proposition~\ref{p:Cayley_undirected}
says that $C S = S$ and $\langle C
\rangle$ is a completely simple
semigroup. Take an arbitrary
element $c$ in $C$. It remains
to verify that $CS = S$, and the
set $C c$ is a left ideal
of $T$ and a left group.

Put $T = \langle C \rangle$. There
exist a group $G$, sets $I$,
$\Lambda$ and $\Lambda \times I$
sandwich matrix $P$ over $G$ such
that $T = M(G; I, \Lambda; P)$.
Besides, $c = (g_c; i_c, \lambda_c)$
for some $g_c \in G$, $i_c \in I$
and $\lambda_c \in \Lambda$.

First, note that $\Cay(T, C)$
coincides with the subgraph of
$\Cay(S, C)$ induced by the set
$T$ of vertices. Therefore it is
also a disjoint union of complete
graphs.

Clearly, $C c$ is equal to the set of
vertices of the connected component of
$\Cay(T, C)$ containing $c$. It follows
from Lemma~\ref{l:connected_component}
that the set of vertices of
the same connected component is
equal to $T c$. Since it is a complete
graph with edges determined by the left
multiplication of the elements of $C$,
we see that it is also equal to the set
$C^2 c$. Since $T C \subseteq T$, we
get $T C c \subseteq T c = C c$.
This means that the set $C c$ is
a left ideal of $T$.

Condition~(iii) of Lemma~\ref{l:Rees_matrix}
implies that $T c = T_{* \lambda_c}$.
Condition~(x) of Lemma~\ref{l:Rees_matrix}
tells us that $T c$ is a left group, as
required.

(vi)$\Rightarrow$(v):
Suppose that condition (vi) holds. Since
implication (vi)$\Rightarrow$(i) has
already been proved, we can also use
(i) and all properties that have
been verified during the proof of~(i).
First of all note that, since
$\Cay(S, C)$ is a union of complete
graphs, it is undirected. Take any
integer $\ell > 1$, any
$k_1, \dots, k_\ell \in \N$, and
pick an arbitrary element $c \in C$.
To prove the assertion, we have
to verify that equality
(\ref{eq:complete:lambda=(|Cc|_1)k1})
is valid.

As it has been shown in the proof of the
implication (i)$\Rightarrow$(vi) above,
$\Cay(S, C)$ is a disjoint union of the
graphs induced by the sets of vertices
$C c s$, for $c \in C$, $s \in S$.
The number of vertices of each of these
subgraphs does not exceed the cardinality
$n = |C c|$. As noted above in the
discussion of equality
(\ref{eq:upper_bound_underlying}),
each complete graph $K_{n}$
has an $L(k_1, \dots, k_\ell)$-distance
labelling with minimum span
$(n - 1) k_1$. It follows that
(\ref{eq:complete:lambda=(|Cc|_1)k1})
holds. This completes the proof.
\end{proofthm}

The following easy lemma is used
repeatedly in the proof of
Theorem~\ref{t:span_combinatorial}.

\begin{lemma}\label{l:e_xx=x}
Let $S$ be a semigroup with a
subset $C$ such that
$\langle C \rangle$ is
a band, $C S = S$, and let
$x \in S$. Then there
exists $e_x \in C$ such that
$e_x x = x$.
\end{lemma}

\begin{proof}
Take any element $x$ in $S$. Since
$C S = S$, it follows that there exist
$e_x \in C$, $f_x \in S$ such that
$e_x f_x = x$. Since $S$ is a band,
$e_x e_x = e_x$. Hence we get
$e_x x = e_x e_x f_x = e_x f_x = x$,
as required.
\end{proof}

\begin{proofthm} {\it of
Theorem~\ref{t:span_combinatorial}.}
Implications
(v)$\Rightarrow$(iv)$\Rightarrow$(iii)$\Rightarrow$(ii)
are obvious. Let us prove implications
(ii)$\Rightarrow$(i) and
(i)$\Rightarrow$(v).

(ii)$\Rightarrow$(i):
Suppose to the contrary that
condition~(ii) holds, but $S$ is
not combinatorial. This means that $S$
contains a subgroup $G$, which has
an element $g$ different from the
identity $e$ of $G$. Take $T = G$
and $C = \{g, g^{-1}\}$. The
Cayley graph $\Cay(T, C)$ is
undirected, because $C = C^{-1}$.
Therefore
inequality~(\ref{eq:comb:lambda_le_F(k1)})
holds, and so condition~(ii) of
Theorem~\ref{t:span_cayley_union_of_complete}
is satisfied. Hence $\Cay(T, C)$ is a
disjoint union of complete graphs,
where it is assumed that the complete
graphs contain all loops. It follows
that $(e, e)$ is an edge of
$\Cay(T, C)$. However, $g e \ne e$ and
$g^{-1} e \ne e$. This contradiction shows
that our assumption was wrong and $S$
must have been combinatorial.

(i)$\Rightarrow$(v):
Suppose that condition (i) holds.
Choose a nonempty subsemigroup $T$
of $S$ and a nonempty subset $C$
of $T$ such that the Cayley graph
$\Cay(T, C)$ is undirected. Take
any integer $\ell > 1$, any
$k_1, \ldots, k_\ell \in \N$,
and fix an arbitrary element $c$
is $C$. We have to verify that
equality~(\ref{eq:comb:lambda=(|Cc|-1)k1})
is valid.

Proposition~\ref{p:Cayley_undirected}
implies that $C T = T$ and the semigroup
$\langle C \rangle = M(H; I, \Lambda; P)$
is completely simple.
Lemma~\ref{l:Rees_matrix}(vii) tells
us that for each $i \in I$,
$\lambda \in \Lambda$, the set
$T_{i\lambda}$ is isomorphic to
$G$ and is a maximal subgroup of
$\langle C \rangle$. However, $T$ is
combinatorial, because it is a
subsemigroup of $S$. Therefore
$|G| = 1$.  It follows that $\langle C
\rangle$ is a rectangular band, and so
it is isomorphic to the direct product
of a left zero band $I$ and a right zero
band $\Lambda$, by
Lemma~\ref{l:rectangular=left_x_right}.

Take an arbitrary element $x =
(g; i, \lambda) \in \langle C \rangle$,
where $i \in I$, $\lambda \in \Lambda$.
We claim that $\langle C \rangle x
= C x$. To prove this, note that it
easily follows from $\langle C \rangle
= I \times \Lambda$ and the definitions
of a right zero band and a left zero band,
that
\begin{equation}\label{eq:pr:<C>x}
\langle C \rangle x = \{ (j, \lambda) \mid j \in I\}.
\end{equation}
For any $j \in I$ and $\mu \in \Lambda$,
we get
$(j, \mu) x \in \langle C \rangle_{j \lambda}
\subseteq \langle C \rangle_{j *}$.
If there existed $j \in I$ such that
$C \cap \langle C \rangle_{j *} = \emptyset$,
then it would follow that
$\langle C \rangle \cap \langle C
\rangle_{j *} = \emptyset$, which would
be a contradiction. Therefore, for each
$j \in I$, there exists $c_j \in C \cap
\langle C \rangle_{j *}$. Hence we have
$c_j x = (j, \lambda)$. It follows that
\begin{equation}\label{eq:pr:Cx}
C x = \{ (j, \lambda) \mid j \in I \}.
\end{equation}
Thus, (\ref{eq:pr:<C>x}) and
(\ref{eq:pr:Cx}) yield
$\langle C \rangle x = C x$.

Second, it also follows from
equalities (\ref{eq:pr:<C>x}) and
(\ref{eq:pr:Cx}) that all cardinalities
$|\langle C \rangle x| = |C x|$ are
equal to the same number $|I|$, and
so they are all equal to
$|C c|$.

Third, (\ref{eq:pr:Cx}) implies that the
subgraph induced by the set $C x$ in
$\Cay(T, C)$ is a complete graph.

Next, we take an arbitrary element $t$ in
$T$, and claim that the subgraph induced
by the set $C t$ in $\Cay(T, C)$ is a
complete graph. Lemma~\ref{l:e_xx=x}
tells us that $e_t t = t$ for some $e_t
\in C$. We have already shown in the
preceding paragraph that $C e_t$ induces
a complete subgraph in $\Cay(T, C)$.
This means that, for any $c_1, c_2 \in C$,
the pair $(c_1 e_t, c_2 e_t)$ is an edge
of $\Cay(T, C)$, and so there exists
$c_{c_1 c_2} \in C$ such that
$c_{c_1 c_2} c_1 e_t = c_2 e_t$. We get
$c_{c_1 c_2} c_1 t = c_{c_1 c_2} c_1 e_t t
= c_2 e_t t = c_2 t$. This means that
$(c_1 t, c_2 t)$ is an edge of the
subgraph induced by $C t$ in $\Cay(T, C)$.
Thus, this subgraph is indeed complete,
with a loop attached to each of its vertices.

Further, suppose that for some elements
$t_1$ and $t_2$ in $T$ the sets $C t_1$
and $C t_2$ are not disjoint. Fix any
element $z \in C t_1 \cap C t_2$, and
consider an arbitrary pair of elements
$z_1 \in C t_1$ and $z_2 \in C t_2$,
where $z_1 = c_1 t_1$, $z_2 = c_2 t_2$,
$c_1, c_2 \in C$. Since $C t_1$ induces
a complete subgraph in $\Cay(T, C)$,
we can find $c' \in C$ such that
$c' z_1 = c' c_1 t_1 = z$. Likewise,
since $C t_2$ induces a complete
subgraph in $\Cay(T, C)$, there exists
$c'' \in C$ such that $c'' z = c_2 t_2
= z_2$. Hence $c'' c' z_1 = z_2$.
Lemma~\ref{l:e_xx=x} implies that
$e_{z_1} z_1 = z_1$ for some $e_{z_1}
\in C$. As we have already proved in the
preceding paragraph, the subgraph
induced by $C e_{z_1}$ in $\Cay(T, C)$
is complete. Hence there exists $\bar{c}
\in C$ such that $\bar{c} e_{z_1} =
c'' c' z_1$. Therefore $\bar{c} z_1 =
z_2$, which means that $(z_1, z_2)$ is
an edge of $\Cay(T, C)$. It follows that
$z_1, z_2 \in C z$. Thus $C t_1 = C t_2
= C z$.

We have shown that any sets of the
form $C t_1$ and $C t_2$, which are
not disjoint, coincide. Since $S$
is finite, it follows that
there exist $n \in \N$ and elements
$t_1, \dots, t_n$ such that $T$ is a
disjoint union of the sets
$C t_1, \dots, C t_n$. We have also
proved that each of the sets $C t_j$,
for $j = 1, \dots, n$, induces a
complete subgraph in $\Cay(T, C)$.
Thus, $\Cay(T, C)$ is a disjoint union
of subgraphs induced by the sets
$C t_1, \dots, C t_n$.

For each $j = 1, \dots, n$,
Lemma~\ref{l:e_xx=x} allows us to find
an element $e_j$ such that $e_j t_j =
t_j$. We get $C t_j = C e_j t_j$.
Therefore the cardinality of the set $C
t_j$ is equal to $|C e_j \cdot t_j|$,
which does not exceed $|C e_j|$. As we
have proved above, $|C e_j| = |C c|$.
Thus, $\Cay(T, C)$ is a disjoint union
of complete graphs each with at most
$|C c|$ vertices. As we have already
noticed a few times during the proofs
above, for every complete graph
$K_m$ with $m$ vertices we have
$\lambda_{k_1, \dots, k_\ell}(K_m) =
(m - 1) k_1$. It follows that
$\lambda_{k_1, \dots, k_\ell}(\Cay(T, C))
= (|C c| - 1) k_1$. This completes the
proof.
\end{proofthm}

\begin{proofthm} {\it of
Theorem~\ref{t:span_right_zero_band}.}
Implications
(v)$\Rightarrow$(iv)$\Rightarrow$(iii)$\Rightarrow$(ii)
are obvious.
Let us prove implications
(i)$\Rightarrow$(v) and
(ii)$\Rightarrow$(i).

(i)$\Rightarrow$(v):
Suppose that $S$ is a right zero band.
Take an arbitrary nonempty subset $C$ of
$S$, any integer $\ell > 1$, and any
positive integers $k_1, \ldots, k_\ell$.
The set of edges of $\Cay(S, C)$ is
equal to $E =\{(s, s) \mid s \in S\}$,
because $c s = s$ for every $c \in C$,
$s \in S$. It follows that $\Cay(S, C)$
is undirected and the function
assigning $1$ to all vertices of
$\Cay(S, C)$ is an
$L(k_1, \dots, k_\ell)$-labelling.
Therefore
$\lambda_{k_1, \dots, k_\ell}(\Cay(S, C)) = 0$.
In the right zero band $S$, we have
$C c = \{c\}$, for all $c \in C$.
Hence $(|C c| - 1) k_1 = 0$, and so
equality (\ref{eq:lambda=(|Cc|-1)k_1})
holds for all $c \in C$. This means
that condition~(v) is satisfied, as
required.

(ii)$\Rightarrow$(i):
Suppose that condition (ii) of
Theorem~\ref{t:span_right_zero_band}
holds. It shows immediately that,
for any subset $C$ of $S$,
condition~(ii) of
Theorem~\ref{t:span_cayley_union_of_complete}
is satisfied. Since
Theorem~\ref{t:span_cayley_union_of_complete}
has already been proved above, we can
apply its condition (i) and conclude that
for every subset $C$ of $S$ we have the
following properties:
$CS = S$, $\langle C \rangle$
is a completely simple semigroup,
and for each $c \in C$ the
set $C c$ is a left group
and a left ideal of
$\langle C \rangle$. Now let us apply
these properties to several different
subsets $C$ of $S$.

First, consider the set $C = S$.
We see that $S = \langle C \rangle =
M(H; I, \Lambda; P)$ is a completely
simple semigroup.

Second, pick any element $c \in S$
and consider the set $C = \{c\}$.
There exist $g \in G$, $i \in I$,
$\lambda \in \Lambda$ such that
$c = (g; i, \lambda)$, and so
$c \in S_{i *}$. As noted above,
$C S = S$. However, condition~(ii) of
Lemma~\ref{l:Rees_matrix} tells us that
$S_{i *}$ is a left ideal of $S$;
whence $c S \subseteq S_{i *}$.
It follows that $S = S_{i *}$, which
means that $I = \{i\}$. Condition~(x)
of Lemma~\ref{l:Rees_matrix} tells us
that $S = S_{i *}$ is a right group.
By Lemma~\ref{l:right-groups}(iii),
$S$ is isomorphic to the direct
product of a right zero band and a
group. To simplify notation, we may
assume that $\Lambda$ is a right
zero band and $S = G \times
\Lambda$.

Third, choose any element $c =
(g, \mu) \in G \times \Lambda = S$,
where $g \in G$, $\mu \in \Lambda$,
and consider the set $C = \{(g, \mu)\}$.
Evidently, $C c = \{ (g^2, \mu)\}$. As
noted above, we know that $C c$ is a
left ideal of $\langle C \rangle$. Hence
$c (g^2, \mu) = (g^3, \mu) \in C c$, and
so $g^2 = g^3$. It follows that $g$ is
equal to the identity $e$ of $G$. Since
$g$ was chosen as an arbitrary element
of $G$, we see that $G = \{e\}$.
Therefore $S \cong \Lambda$ is a
right zero band. Thus, condition~(i) of
Theorem~\ref{t:span_right_zero_band}
holds. This completes the proof.
\end{proofthm}



\small
\begin{thebibliography}{99}

\bibitem{AbawajyKelarev2013combinatorial}
J. Abawajy and A. V. Kelarev,
Classification systems based on combinatorial semigroups,
\textit{Semigroup Forum},
{\bf 86} (2013), 603--612.

\bibitem{AlonMohar2002}
N. Alon and B. Mohar,
The chromatic number of graph powers,
\textit{Combin. Probab. Comput.},
{\bf 11} (2002), 1--10.

\bibitem{BrankovicMillerPlesnikRyanSiran1998}
L. Brankovic, M. Miller, J. Plesnik, J. Ryan and J. Siran,
A note on constructing large Cayley graphs of given degree and diameter by voltage assignment,
\textit{Electronic J. Combinatorics},
{\bf 5} (1998), 147--158.

\bibitem{Calamoneri2006}
T. Calamoneri,
The $L(h, k)$-labelling problem: a survey and annotated bibliography,
\textit{The Computer Journal},
{\bf 49} (2006), 585--608.

\bibitem{ChangLuSanmingZhou2007}
G. J. Chang, C. Lu and S. Zhou,
No-hole $2$-distant colorings for Cayley graphs on finitely generated abelian groups,
\textit{Discrete Mathematics},
{\bf 307} (2007), 1808--1817.

\bibitem{ChangLuSanmingZhou2009}
G. J. Chang, C. Lu, and S. Zhou,
Distance-two labellings of Hamming graphs,
\textit{Discrete Applied Mathematics},
{\bf 157} (2009), 1896--1904.

\bibitem{Denes1967}
J. D\'{e}nes,
Connections between transformation semigroups and graphs,
\textit{Theory of Graphs},
Gordon and Breach, New York~-- Paris (1967), pp. 93--101.

\bibitem{FanZeng2007}
S. Fan and Y. Zeng,
On Cayley graphs of bands,
\textit{Semigroup Forum},
{\bf 74} (2007), 99--105.

\bibitem{GaoLiuLuo2011}
X. Gao, W. Liu and Y. Luo,
On Cayley graphs of normal bands,
\textit{Ars Combinatoria},
{\bf 100} (2011), 409--419.

\bibitem{HaoGaoLuo2011}
Y. Hao, X. Gao and Y. Luo,
On the Cayley graphs of Brandt semigroups,
\textit{Communications in Algebra},
{\bf 39} (2011), 2874--2883.

\bibitem{Howie:fst}
J. M. Howie,
\textit{Fundamentals of Semigroup Theory}
(Clarendon Press, Oxford, 1995).

\bibitem{Jiang2004}
Z. Jiang,
An answer to a question of Kelarev and
Praeger on Cayley graphs of semigroups,
\textit{Semigroup Forum},
{\bf 69} (2004) 457--461.

\bibitem{Kelarev2002undirected}
A. V. Kelarev, On undirected Cayley graphs,
\textit{Australasian Journal of Combinatorics},
{\bf 25} (2002), 73--78.

\bibitem{Kelarev2004cayley}
A. V. Kelarev,
Labelled Cayley graphs and minimal automata,
\textit{Australasian J. Combinatorics},
\textbf{30} (2004), 95--101.

\bibitem{Kelarev2004inverse}
A. V. Kelarev,
On Cayley graphs of inverse semigroups,
\textit{Semigroup Forum},
\textbf{72} (2006), 411--418.

\bibitem{Kelarev:gaa}
A. V. Kelarev,
\textit{Graph Algebras and Automata}
(Marcel Dekker, New York, 2003).

\bibitem{KelarevPraeger2003}
A. V. Kelarev and C. E. Praeger,
On transitive Cayley graphs of groups and semigroups,
\textit{European J. Combinatorics},
\textbf{24} (2003), 59--72.

\bibitem{KelarevRyanYearwood2009cayley}
A. V. Kelarev, J. Ryan and J. Yearwood,
Cayley graphs as classifiers for data mining: The influence of asymmetries,
\textit{Discrete Mathematics},
\textbf{309} (2009), 5360--5369.

\bibitem{KelarevYearwoodWatters2011combinatorial}
A. V. Kelarev, J. L. Yearwood and P. A. Watters,
Optimization of classifiers for data mining based on combinatorial semigroups,
\textit{Semigroup Forum},
{\bf 82} (2011), 242--251.

\bibitem{KhosraviKhosravi2012}
B. Khosravi and B. Khosravi,
A characterization of Cayley graphs of Brandt semigroups,
\textit{Bulletin of the Malaysian Mathematical Sciences Society},
{\bf 35} (2012), 399--410.

\bibitem{KhosraviKhosravi2013}
B. Khosravi and B. Khosravi,
On Cayley graphs of semilattices of semigroups,
\textit{Semigroup Forum},
{\bf 86} (2013), 114--132.

\bibitem{KhosraviKhosravi2014}
B. Khosravi and B. Khosravi,
On combinatorial properties of bands,
\textit{Communications in Algebra},
{\bf 42} (2014), 1379--1395.

\bibitem{KhosraviKhosraviKhosravi2014}
B. Khosravi, B. Khosravi and B. Khosravi,
On color-automorphism vertex transitivity of semigroups,
\textit{European Journal of Combinatorics},
{\bf 40} (2014), 55--64.

\bibitem{KhosraviKhosraviKhosravi2015}
B. Khosravi, B. Khosravi, B. Khosravi,
On the Cayley D-saturated property of semigroups,
\textit{Semigroup Forum},
DOI~10.1007/s00233-015-9716-2,
to appear in 2015.

\bibitem{KhosraviMahmoudi2010}
B. Khosravi and M. Mahmoudi,
On Cayley graphs of rectangular groups,
\textit{Discrete Mathematics},
{\bf 310} (2010), 804--811.

\bibitem{KingRasSanmingZhou2010}
D. King, C. J. Ras and S. Zhou,
The $L(h, 1, 1)$-labelling problem for trees,
\textit{European Journal of Combinatorics},
{\bf 31} (2010), 1295--1306.

\bibitem{KingLiSanmingZhou2014}
D. King, K. Y. Li and S. Zhou,
Linear and cyclic distance-three labellings of trees,
\textit{Discrete Applied Math.},
{\bf 178} (2014), 109--120.

\bibitem{Knauer2011}
U. Knauer,
\textit{Algebraic Graph Theory. Morphisms, Monoids and Matrices}
(Walter de Gruyter \& Co., Berlin, 2011).

\bibitem{LakshmivarahanJwoDhall1993}
S. Lakshmivarahan, J.-S. Jwo and S. K. Dhall,
Symmetry in interconnection networks based on Cayley graphs of permutation groups: a survey,
\textit{Parallel Computing},
{\bf 19} (1993), 361--407.

\bibitem{LiPraeger1999}
C. H. Li and C. E. Praeger,
On the isomorphism problem for finite Cayley graphs of bounded valency,
\textit{European J. Combinatorics},
{\bf 20} (1999), 279--292.

\bibitem{LuoHaoClarke2011}
Y. Luo, Y. Hao and G. T. Clarke,
On the Cayley graphs of completely simple semigroups,
\textit{Semigroup Forum},
{\bf 82} (2011), 288--295.

\bibitem{PanmaChiangmaiKnauerArworn2006}
S. Panma, N. Chiangmai, U. Knauer and Sr. Arworn,
Characterizations of Clifford semigroup digraphs,
\textit{Discrete Math.},
{\bf 306} (2006), 1247--1252.

\bibitem{Praeger1999}
C. E. Praeger,
Finite normal edge-transitive Cayley graphs,
\textit{Bulletin of the Australian Mathematical Society},
{\bf 60} (1999), 207--220.

\bibitem{WangLi2013}
S. Wang and Y. Li,
On Cayley graphs of completely 0-simple semigroups,
\textit{Central European J. Mathematics},
{\bf 11} (2013), 924--930.

\bibitem{Zelinka1981}
B. Zelinka,
Graphs of semigroups,
\textit{Casopis. Pest. Mat.},
{\bf 106} (1981), 407--408.

\bibitem{SanmingZhou2006}
S. Zhou,
Labelling Cayley graphs on abelian groups,
\textit{SIAM J. Discrete Math.},
{\bf 19} (2006), 985--1003.

\bibitem{SanmingZhou2008}
S. Zhou,
A distance labelling problem for hypercubes,
\textit{Discrete Applied Mathematics},
{\bf 156} (2008), 2846--2854.


\end{thebibliography}
\end{document}